\documentclass[11pt, reqno]{amsart}
\usepackage[utf8]{inputenc}

\usepackage[nospace,noadjust]{cite}
\usepackage{relsize}
\usepackage{cite}
\usepackage{color}
\usepackage[dvipsnames]{xcolor}

\usepackage{tikz-cd}
\usetikzlibrary{calc,fit,matrix,arrows,automata,positioning}
\usetikzlibrary{decorations.pathreplacing,angles,quotes}
\usepackage{colortbl}
\usepackage{hyperref, enumitem}
\hypersetup{
  colorlinks   = true, 
  urlcolor     = blue, 
  linkcolor    = Purple, 
  citecolor   = red 
}
\usepackage{amsmath,amsthm,amssymb}

\usepackage[margin=2.5cm]{geometry}

\usepackage{stmaryrd}

\usepackage{array}
\newcolumntype{x}[1]{>{\centering\arraybackslash\hspace{0pt}}p{#1}}

\theoremstyle{definition}
\newtheorem{theorem}{Theorem}[section]
\newtheorem{definition}[theorem]{{{Definition}}}
\newtheorem{example}[theorem]{{{Example}}}

\newtheorem{remark}[theorem]{{{Remark}}}

\newtheorem{corollary}[theorem]{{{Corollary}}}
\newtheorem{proposition}[theorem]{{{Proposition}}}
\newtheorem{lemma}[theorem]{{{Lemma}}}

\newtheorem{question}[theorem]{{{Question}}}

\newcommand{\pgm}{\text{PG}(k-1, q^{m})}

\newcommand{\Fqn}{\mathbb{F}_{q}^{n}}
\newcommand{\nkdm}{[n,k,d]_{q^m/q}}
\newcommand{\numberset}{\mathbb}

\newcommand{\C}{\mathcal{C}}

\newcommand{\F}{\numberset{F}}
\newcommand{\Mat}{\mbox{Mat}}
\newcommand{\mS}{\mathcal{S}}

\newcommand{\mC}{\mathcal{C}}

\newcommand{\mB}{\mathcal{B}}

\newcommand{\mM}{\mathcal{M}}

\newcommand{\mU}{\mathcal{U}}

\newcommand{\mV}{\mathcal{V}}

\newcommand{\mH}{\mathcal{H}}

\newcommand{\wt}{\textnormal{wt}}

\newcommand{\Fq}{\F_q}
\newcommand{\Fm}{\F_{q^m}}

\newcommand{\rk}{\textnormal{rk}}

\DeclareMathOperator{\GL}{GL}

\DeclareMathOperator{\PG}{PG}

\newcommand{\Fqmk}{\mathbb{F}_{q^{m}}^{k}}
\newcommand{\Fqmn}{\mathbb{F}_{q^{m}}^{n}}
\newcommand{\Fqm}{\mathbb{F}_{q^{m}}}

\newcommand{\rowspan}{\text{rowspan}}
\newcommand{\nkdqm}{[n, k, d]_{q^{m}/q}}

\title{Linear rank-metric intersecting codes}
\usepackage[foot]{amsaddr}
\author{Daniele Bartoli$^1$}
\author{Martino Borello$^{2,3}$}
\author{Giuseppe Marino$^4$}
\author{Martin Scotti$^2$}
\address{$^1$Universit\`a degli Studi di Perugia, Italy.}
\address{$^2$Universit\'e Paris 8, Laboratoire de G\'eom\'etrie, Analyse et Applications, LAGA, Universit\'e Sorbonne Paris Nord, CNRS, UMR 7539, France.}
\address{$^3$INRIA, France.}
\address{$^4$Universit\`a di Napoli Federico II, Italy.}

\email{daniele.bartoli@unipg.it} \email{martino.borello@univ-paris8.fr}
\email{giuseppe.marino@unina.it}
\email{martin.scotti@etud.univ-paris8.fr}
\thanks{D.~B. and G.~M. thank the Italian National Group for Algebraic and Geometric Structures and their Applications (GNSAGA—INdAM) which supported the research. They would also like to thank the Departement of Mathematics and the LAGA laboratory for their support and the kind hospitality in their stimulating research environment}
\thanks{M.~B. and M.~S. are partially supported by the ANR-21-CE39-0009 - BARRACUDA (French \emph{Agence Nationale de la Recherche}) and PHC Galil\'ee 2024 (project 50424WM)}

\begin{document}

\begin{abstract} 
In this paper we introduce and investigate rank-metric intersecting codes, a new class of linear codes in the rank-metric context, inspired by the well-studied notion of intersecting codes in the Hamming metric. A rank-metric code is said to be intersecting if any two nonzero codewords have supports intersecting non trivially. We explore this class from both a coding-theoretic and geometric perspective, highlighting its relationship with minimal codes, MRD codes, and Hamming-metric intersecting codes. We derive structural properties, sufficient conditions based on minimum distance, and geometric characterizations in terms of 2-spannable $q$-systems. We establish upper and lower bounds on code parameters and show some constructions, which leave a range of unexplored parameters. Finally, we connect rank-intersecting codes to other combinatorial structures such as $(2,1)$-separating systems and frameproof codes. 
\end{abstract}
\maketitle


\bigskip

\section*{Introduction}

Intersecting codes in the Hamming metric are linear codes in which any two nonzero codewords share at least one coordinate in their support, where the support of a codeword is the set of nonzero coordinates. These codes were originally introduced in the foundational works \cite{miklos1984linear,katona1983minimal} and have since been studied extensively in the literature (e.g., \cite{borello2025geometry,CZ,retter1989intersecting,cohnen2003intersecting}), primarily within the binary setting. In the binary case, intersecting codes coincide with minimal codes, a class that has garnered significant attention over the past two decades. Intersecting codes are relevant to various practical scenarios, including communication over AND channels, applications in secret sharing schemes \cite{massey1993minimal} and oblivious transfer protocols \cite{brassard2002oblivious}, and connections to related combinatorial structures such as frameproof codes \cite{blackburn2003frameproof} and 
$(2,1)$-separating systems \cite{randriambololona20132}. They are also related to zero-sum problems in additive combinatorics and factorization theory in Dedekind domains \cite{plagne2011application,borello2025geometry}.

In recent years, particularly with the advent of network coding \cite{silva2008rank,koetter2008coding} and the renewed interest in code-based cryptography \cite{alagic2020status}, the class of rank-metric codes has become a central topic of intense mathematical investigation. Linear rank-metric codes can be studied using geometric tools, by identifying them with linear sets in a projective space. This connection links, for instance, optimal codes with respect to the rank metric—known as MRD codes—to sets with good intersection properties with hyperplanes, called scattered linear sets with respect to hyperplanes \cite{sheekey2019linear}. In \cite{rankminimal_ABNR}, the authors clarify the relationship between another family of rank-metric codes, namely minimal codes, and linear sets that are strong blocking sets.

\medskip

In this paper, inspired by the importance of intersecting codes in the Hamming metric and the interest for linear rank-metric codes, we introduce for the first time the notion of intersecting codes in the rank metric, and we explore this concept from both a coding-theoretic and geometric perspective. 

\medskip

A rank-metric code is said to be \emph{intersecting} if the supports of any two nonzero codewords intersect nontrivially. In the first part of the paper, we highlight some basic properties that such rank-metric intersecting codes must satisfy, as well as their connections to Hamming-metric intersecting codes and to MRD codes, particularly in the regime where the code length is smaller than the extension degree of the field. We also compare rank-intersecting codes with minimal codes, showing that these are two distinct classes of codes—although there do exist codes that are both rank-intersecting and minimal. A simple yet fundamental property is the following: if the minimum distance of a code (that is, the smallest dimension of the support of a nonzero codeword) exceeds half the code's length, then the code is clearly rank-intersecting. This condition, which also holds in the Hamming-metric case (see for instance \cite{borello2025geometry}), depends solely on the code's weight distribution and serves as a sufficient condition. It can be viewed as an analogue of the well-known Ashikhmin–Barg condition for minimal codes \cite{ashikhmin1998minimal}.

Later in the paper, we show that a rank-intersecting code corresponds to a $q$-system (the vector space underlying a linear set) that is not 2-spannable, meaning it is not generated by the sum of its intersections with any two hyperplanes. This property provides a geometric perspective on the code, allowing us to derive further results—particularly bounds on parameters—and to develop new constructions. Notably, this geometric viewpoint enables the construction of intersecting codes that do not necessarily satisfy the aforementioned sufficient condition. For example, we relate intersecting codes to certain geometric structures known as $t$-clubs.
Regarding bounds, in addition to proving that the minimum distance must always be at least equal to the dimension, we show that a (nondegenerate) rank-intersecting code of dimension $k$ over $\mathbb{F}_{q^m}$ can exist only if its length $n$ satisfies the inequality $$2k - 1 \leq n \leq 2m - 3.$$ Currently, known constructions exist only for lengths in the range $2k - 1 \leq n \leq 2m - 2k + 1$, leaving a gray area that still needs to be explored.

In the final part of the paper, we examine the connections between rank-intersecting codes and other combinatorial structures, which we reinterpret in the rank-metric setting. Classical $(2,1)$-\emph{separating codes} (see for example \cite{randriambololona20132}) are codes in which no three vectors satisfy the triangle inequality with equality (with respect to the Hamming metric). Their linear version corresponds to Hamming-metric intersecting codes. Naturally, this definition can also be extended to the rank-metric context. In this case, we show that the class of linear rank-metric $(2,1)$-separating codes contains both rank-intersecting codes and minimal codes. The classical question about the maximum size of such codes (for a fixed length) finds a partial answer in Gabidulin codes of length $2k - 1$. Lastly, we explore the connection with \emph{frameproof codes} (see \cite{blackburn2003frameproof} for the Hamming-metric context). In particular, we provide a definition of descendants and prove that $2$-frameproof linear rank-metric codes coincide with rank-intersecting codes. It would be very interesting to explore potential applications of these codes to fingerprinting in the context of network coding.

\bigskip

\noindent \textbf{Outline.} In Section \ref{sec:background}, we recall the basic notions of rank-metric codes, MRD codes, and their geometric interpretation through linear sets and 
$q$-systems. In Section \ref{sec:rankinter}, we introduce rank-metric intersecting codes, establish their main properties, and compare them with minimal codes. Section \ref{sec:geom} provides a geometric interpretation of these codes via $2$-spannable 
$q$-systems, leading to new constructions. In Section \ref{sec:bounds}, we derive bounds on the parameters of rank-intersecting codes.
Section \ref{sec:separ} revisits the concept of 
$(2,1)$-separating codes in the rank-metric setting, showing how it encompasses both rank-intersecting and minimal codes. Finally, in Section \ref{sec:frameproof}, we study the connection with frameproof codes, proving that 2-frameproof linear rank-metric codes coincide with rank-intersecting codes.

\bigskip

\section{Background}\label{sec:background}

The aim of this section is to present all notions and notations which will be used in the rest of the paper.

\subsection{Rank-metric codes}
Let $\Gamma=\{\gamma_1,\ldots,\gamma_m\}$ be an $\Fq$-basis of $\Fqm$. For any vector $x=(x_1,\ldots,x_n) \in \Fqmn$, we denote by
$$\Mat_{\Gamma}(x) \in \F_q^{m\times n}$$ 
the matrix whose coefficients $a_{i,j}$ are such that, for all $i\in\{1,\ldots,n\}$,
$$x_i=\sum_{j=1}^m a_{i,j}\gamma_j,$$
that is the columns of $\Mat_{\Gamma}(x)$
represent the coordinates of $x$ in the basis $\Gamma$.
Note that the rank of $\Mat_{\Gamma}(x)$ does not depend on the choice of the basis $\Gamma$. Indeed, if $\Gamma'$ is another $\Fq$-basis of $\Fqm$, then $\Mat_{\Gamma'}(x) = A \cdot \Mat_{\Gamma}(x)$, where $A\in {\rm GL}(m,q)$. Hence, the following definition makes sense.

\begin{definition}\label{Def:rank}
The \emph{rank} of $x \in \Fqmn$, noted $\rk(x)$, is the rank of the matrix $\Mat_{\Gamma}(x)$.
\end{definition}

The rank function $x\mapsto \rk(x)$ induces a metric, for which the distance between $x,y\in \Fqmn$ is $\rk(x-y)$. This is the so-called \emph{rank metric}.

For the same reason as above, the rowspan of the rows of $\Mat_{\Gamma}(x)$ does not depend on $\Gamma$. 

\begin{definition}
The \emph{support} of  $x \in \Fqmn$ is $\sigma(x) = \rowspan(\Mat_{\Gamma}(x))$. \end{definition}

Clearly $\rk(x)=\dim_{\Fq} \sigma(x)$. Note that this is not the only possible definition of support (see for example \cite{gorla2021rank}), but it has the nice property of being invariant under nonzero scalar multiplication, that is $\sigma(\lambda x)=\sigma(x)$ for all $\lambda\in\Fqm^\ast$.\\

The main objects of these papers are the $\Fqm$-linear subspaces of $\Fqm^n$, endowed with the rank metric. In the literature, they are called in different ways, in order to distinguish them from the subspaces of matrices. In this paper, we will simply refer to them as rank-metric codes. 

\begin{definition}
A \emph{rank-metric code} $\C$ is a vector subspace of $\Fqmn$, endowed with the rank metric. Its \emph{length} is $n$, its \emph{dimension} is $k = \dim_{\Fqm}(\C)$, and its \emph{minimum distance} is ${\rm d}(\C) := \min_{c\in \C \setminus \{0\}} \rk(c)$.  
We say that $\C$ is an $\nkdm$ code.
A \emph{generator matrix} of $\C$ is an $k \times n$ matrix with coefficients in $\Fqm$ such that $\C = \rowspan(G)$.
\end{definition}

\begin{definition}
Two rank-metric codes $\C$ and $\C'$ of length $n$ over $\Fqm$ are said to be \emph{equivalent} if there exists $A \in \text{GL}(n, q)$ such that
$$\C' = \C \cdot A = \{ c \cdot A : c\in \C\}.$$
\end{definition}

As in the Hamming metric, two equivalent codes in the rank metric necessarily have the same parameters.

\begin{definition}
 An $\nkdm$ code $\mC$ is (\emph{rank-})\emph{degenerate} if $\sigma(\mC):=\sum_{c\in \C}\sigma(c)\neq \F_q^n$.
 We say that $\mC$ is (\emph{rank-})\emph{nondegenerate} if it is not degenerate. 
\end{definition}

\begin{proposition}[\cite{rankminimal_ABNR}]\label{prop:nondeg}
An $\nkdqm$ code is nondegenerate if and only if
the $\F_q$-span of the columns of any generator matrix of $\C$ has $\Fq$-dimension $n$.
\end{proposition}

In the next sections, we will often see the following family of codes, introduced in \cite{rankminimal_ABNR}.

\begin{definition}
An $\nkdqm$ code is minimal if for every $c,c'\in \C$, $\sigma(c)\subseteq \sigma(c')$ implies $c'=\lambda c$ for some $\lambda\in \Fqm$.
\end{definition}

\subsection{MRD codes}

An analogue of the Singleton bound holds in the rank metric (see for example \cite{gorla2021rank}).

\begin{theorem}[Singleton bound in rank metric]
Let $\C$ be an $\nkdm$ code. Then
$$km \leq \max(m,n) \cdot (\min(n,m)-d+1).$$
\end{theorem}

Codes that achieve this bound are called \emph{MRD codes} (maximum rank distance codes). While the study of MRD codes in general is an active area of research, Gabidulin codes provide a simple example of such codes if $n\leq m$. They are the analogue of Reed-Solomon codes in the rank metric.

We begin by defining linearized polynomials.

\begin{definition}
A \emph{linearized polynomial} of $q$-\emph{degree} $k$ is a polynomial $P(X)\in \Fqm[X]$ of the form
$$P(X)=p_{0} + p_{1}X^{q} + p_{2} X^{q^{2}} + \dots + p_{k} X^{q^{k}}.$$
\end{definition}

Let $\mathcal{L}_{q^m}[X]$ be the set of linearized polynomials with coefficients in $\mathbb{F}_{q^m}$.

For $n \leq m$, let $a_{1}, \dots, a_{n} \in \Fqm$ be $\Fq$-linearly independent elements. A Gabidulin code of length $n$ and dimension $k$ is defined as
$${\rm Gab}_{n, k}(a_1,\ldots,a_n)= \{ (P(a_{1}), \dots, P(a_{n})) : P \in \mathcal{L}_{q^m}[X], \text{ of } q\text{-degree} < k \}.$$

A generator matrix of ${\rm Gab}_{n, k}(a_1,\ldots,a_n)$ is
$$G = \begin{pmatrix}
a_{1} & \dots & a_{n} \\
a_{1}^{q} & \dots & a_{n}^{q} \\
\vdots & \vdots & \vdots \\
a_{1}^{q^{k-1}} & \dots & a_{n}^{q^{k-1}}
\end{pmatrix}.$$
This type of matrix is called a \emph{Moore matrix}, and it has full rank if and only if $a_{1}, \dots, a_{n}$ are $\Fq$-linearly independent. The Gabidulin code ${\rm Gab}_{n, k}(a_1,\ldots,a_n)$ has parameters $[n, k, n-k+1]_{q^{m}/q}$ and it is an MRD code.

\subsection{Linear sets and $q$-systems}

Let $\mC$ be a nondegenerate $[n,k,d]_{q^m/q}$ code. Let $G$ denote a generator matrix of $\mC$, and let $g_1, \dots, g_n \in \F_{q^m}^k$ be its column vectors. Define the $\F_q$-linear subspace
\[
\mU_G := \langle g_1, \ldots, g_n \rangle_{\F_q},
\]
which has $\F_q$-dimension $n$ (by Proposition \ref{prop:nondeg}), and satisfies $\langle \mU_G \rangle_{\F_{q^m}} = \F_{q^m}^k$. This subspace $\mU_G$ is naturally referred to as the $[n,k]_{q^m/q}$-\emph{system} (or simply $q$-\emph{system}) associated with $\mC$.  

Two $[n,k]_{q^m/q}$ systems $\mU$ and $\mU'$ are said to be \emph{equivalent} if there exists an $\F_{q^m}$-linear isomorphism $\varphi: \F_{q^m}^k \to \F_{q^m}^k$ such that $\varphi(\mU) = \mU'$. It is clear that if $G$ and $G'$ are two generator matrices of the same code $\mC$, then the corresponding systems $\mU_G$ and $\mU_{G'}$ are equivalent. For this reason, we may, with slight abuse of notation, denote the system simply by $\mU$, omitting explicit reference to $G$. Furthermore, it can be easily shown that equivalent codes yield equivalent systems (see the appendix of~\cite{rankminimal_ABNR} for a detailed discussion).

The $q$-system associated with a rank-metric code is intimately linked with certain geometric configurations known as \emph{linear sets}. These structures were first introduced by Lunardon in~\cite{lunardon1999normal} in the context of constructing blocking sets, and have since been widely studied. For a comprehensive overview, see~\cite{polverino2010linear}.

\begin{definition}
Let $\mU$ be an $\F_q$-subspace of $\F_{q^m}^k$ of $\mathbb{F}_q$-dimension $k$. The $\F_q$-\emph{linear set} of rank $n$ associated to $\mU$ is defined as
\[
L_\mU := \left\{ \langle u \rangle_{\F_{q^m}} \;: \; u \in \mU \setminus \{0\} \right\} \subseteq \PG(k-1, q^m).
\]
Two such linear sets are called \emph{equivalent} if their underlying systems are $P\Gamma L$-equivalent.
\end{definition}

\begin{remark}
The classical definition of a linear set does not require that $\langle \mU \rangle_{\F_{q^m}} = \F_{q^m}^k$. If the dimension over $\F_{q^m}$ is $h < k$, we may assume $\mU \subseteq \F_{q^m}^h$ without loss of generality, and consider $L_\mU$ as a subset of $\PG(h-1, q^m)$.
\end{remark}

\begin{definition}
Let $\mathcal{W}$ be an $\F_{q^m}$-subspace of $\F_{q^m}^k$, and denote by $\Lambda = \PG(\mathcal{W}, \F_{q^m})$ the corresponding projective subspace. The \emph{weight of $\Lambda$ in $L_\mU$} is defined as
\[
\wt_\mU(\Lambda) := \dim_{\F_q}(\mU \cap \mathcal{W}).
\]
In the same way, we may define the weight of $\mathcal{W}$ in $\mU$.
\end{definition}

There is an important relation between the rank metric and the geometry of the linear set or of the $q$-system associated to a rank-metric code (see, e.g., \cite{rankminimal_ABNR}): 
\begin{equation}\label{eq:weight}
   \rk(uG)=n-\dim_{\Fq}(\mU_G\cap \langle u\rangle^\perp)=n-\wt_{\mU_G}(\langle u\rangle^\perp).
\end{equation}
This means that the rank metric can be inferred by studying intersections of the linear set with hyperplanes.

\noindent For any system $\mU$, the associated linear set satisfies
\[
|L_\mU| \leq \frac{q^n - 1}{q - 1}.
\]
A linear set attaining equality in this bound is called \emph{scattered} (and $U$ is said to be \emph{scattered}). Equivalently, $L_\mU$ is scattered if and only if every point $P \in L_\mU$ has weight $\mathrm{wt}_\mU(P) = 1$. The notion of scatteredness can be extended, as introduced in~\cite{sheekey2020rank}, to define \emph{$h$-scattered linear sets} (or \emph{$h$-scattered subspaces}), which are scattered with respect to $\F_{q^m}$-vector subspaces of dimension $h$. In particular, $L_\mU \subseteq \PG(k-1, q^m)$ is said to be \emph{scattered with respect to hyperplanes} if $\mathrm{w}_\mU(\mathcal{H}) \leq k - 1$ for every hyperplane $\mathcal{H}$ in $\PG(k-1, q^m)$. In the regime $n\leq m$, by \eqref{eq:weight} this last property is equivalent to the code being MRD.

\bigskip

\section{Rank-metric intersecting codes}\label{sec:rankinter}

As we mentioned in the introduction, intersecting codes are well-studied objects in the literature, for their connections with many other fields of mathematics and for their applications in the real life. It seems quite natural to define their analogue in the rank metric.

\begin{definition} \label{def:rank_intersecting}
A code $\C$ is \emph{rank-metric intersecting} if for any $ c, c' \in \C \setminus \{0\}$ we have
$$\sigma(c) \cap \sigma(c') \neq \{0\}.$$
\end{definition}

\begin{proposition}\label{prop:equivalence}
Let $\C$ be a rank-metric intersecting code, and let $\C'$ be a code equivalent to $\C$. Then $\C'$ is also rank-metric intersecting.
\end{proposition}

\begin{proof}
Let us first note that since $\C'$ is equivalent to $\C$, there exists $A \in \GL(n,q)$ such that $\C' = \{c \cdot A : c \in \C\}$. Let $c, c' \in \C$. Since $\C$ is intersecting, there exists $x \in \sigma(c) \cap \sigma(c')$, $x\neq 0$. On the other hand, $c\cdot A$ and $c'\cdot A$ are codewords of $\C'$. Now, $\sigma(c\cdot A) = \sigma(c) \cdot A$, so that $x \cdot A \in \sigma(c \cdot A)$ and $x \cdot A \in \sigma(c'\cdot A)$. Hence $\sigma(c\cdot A) \cap \sigma(c'\cdot A) \neq \{0\}$, which implies that $\C'$ is intersecting.
\end{proof}

\begin{proposition}\label{prop:rankinthammingint}
If a code is rank-metric intersecting, then it is also Hamming-metric intersecting.
\end{proposition}

\begin{proof}
Let suppose that a code $\C$ is not Hamming-metric intersecting. This means that there exists $c,c'\in \C$ such that
$$\{i:c_i\neq 0\} \cap \{i:c'_i\neq 0\}=\varnothing.$$
Now, $\sigma(c)\subseteq \langle e_j:j\in \{i:c_i\neq 0\}\rangle$ and $\sigma(c')\subseteq \langle e_j:j\in \{i:c'_i\neq 0\}\rangle$,
where $e_1,\ldots,e_n$ is the canonical basis of $\F_q^n$. Hence, $\sigma(c)\cap \sigma(c')=\{0\}$, a contradiction.
\end{proof}

\begin{corollary}\label{Cor:Nec}
If an $\nkdm$ code is rank-metric intersecting, then $n \geq 2k - 1$.
\end{corollary}

\begin{proof}
This is a direct consequence \cite[Theorem 3.2]{borello2025geometry}.
\end{proof}

We give now a simple sufficient condition for a code to be rank-metric intersecting.

\begin{theorem}\label{thm:sufficient}
Let $\C$ be an $\nkdm$ code. If $2d > n$, then $\C$ is rank-metric intersecting.
\end{theorem}

\begin{proof}
Let $c, c' \in \C \setminus \{0\}$. It is clear that $\dim_{\Fq} \sigma(c) \geq d$ and $\dim_{\Fq} \sigma(c') \geq d$. Therefore, since $\sigma(c)$ and $\sigma(c')$ are subspaces of $\Fqn$, and since $2d > n$, we have $\sigma(c) \cap \sigma(c') \neq \{0\}$ by Grassmann's formula, and thus $\C$ is rank-metric intersecting.
\end{proof}

\begin{example}
Let $\C$ be a $[3,2,2]_{8/2}$ code generated by
$$G=\begin{pmatrix}
    1 & \alpha & \alpha^2\\
    1 & \alpha^2 & \alpha^4
\end{pmatrix},$$
where $\alpha$ is a primitive element in $\F_8$. This is actually ${\rm Gab}_{3,2}(1,\alpha,\alpha^2)$. The support of any nonzero codeword is subspace of $\F_2^3$ of dimension at least $2$. According to Grassmann's formula, any two such subspaces must intersect in a subspace of dimension at least $1$, so that $\C$ is rank-metric intersecting
\end{example}

\begin{remark}\label{rmk:nonnecessary}
The condition of Theorem \ref{thm:sufficient} is only sufficient and in general not necessary. For example, the $[4, 2, 2]_{16/2}$ code with
generator matrix:
$$G=\begin{pmatrix}
    \alpha&\alpha^2&\alpha^3&0\\
    0&0&1&\alpha
\end{pmatrix},$$
where $\alpha^4=\alpha+1$, is rank-metric intersecting even if its minimum distance is not greater than $2$.
\end{remark}

For MRD codes the condition of Theorem \ref{thm:sufficient} is also necessary.

\begin{proposition}
Let $n \leq m$. An MRD code is rank-metric intersecting if and only if $2d > n$. 
\end{proposition}

\begin{proof}
Let $\C$ be an $[n, k, n-k+1]_{q^{m}/q}$ code.  

If $2d > n$, then the code is clearly rank-metric intersecting by Theorem~\ref{thm:sufficient}.

If the code is rank-metric intersecting, then by Corollary \ref{Cor:Nec} we have $n>2k-1$, or equivalently $d \geq k$, which implies $2d \geq k + d = n-k + k + 1 = n+1$, hence $2d > n$.
\end{proof}

\begin{remark}
    Note that in the case of MRD codes, the condition $2d > n$ is equivalent to $n \geq 2k - 1$.
\end{remark}

\begin{remark}\label{rmk:min-int}
In the rank metric, the comparison between intersecting codes and minimal codes is more subtle than in the Hamming metric.

Indeed, while every minimal code is intersecting in the Hamming metric, this is certainly not true in the rank metric. The best example of this phenomenon is the $[km,k,m]_{q^m/q}$ simplex code, whose generator matrix is, up to equivalence,
$$
G = \begin{pmatrix}
I_{k} & \alpha\cdot I_{k} & \alpha^{2}\cdot I_{k} & \cdots & \alpha^{m-1}\cdot I_{k} \\
\end{pmatrix},$$
with $\alpha$ a primitive element of $\Fqm$ and $I_k$ the $k\times k$ identity matrix.
All codewords of the simplex code have the same weight \cite{rankminimal_ABNR, randrianarisoa2020geometric}, which implies that it is minimal. However, the simplex code is clearly not rank-metric intersecting for $k\geq 2$: if $e_1,e_2$ are the first elements of the canonical basis of $\Fqm^k$, it is easy to observe that 
$$\sigma(e_1G)\cap \sigma(e_2G)=\{0\}.$$
There exist minimal codes which are also rank-metric intersecting: the $[9,3,5]_{128/2}$ code with generator matrix
$$G:=\begin{pmatrix}
1&\alpha&\alpha^2&\alpha^3&\alpha^4&\alpha^5&\alpha^6&0&0\\
1&\alpha^2&\alpha^4&\alpha^6&\alpha^8&\alpha^{10}&\alpha^{12}&1&0\\
1&\alpha^4&\alpha^8&\alpha^{12}&\alpha^{16}&\alpha^{20}&\alpha^{24}&0&1\\
\end{pmatrix},$$
with $\alpha$ a primitive element of $\F_{128}$, whose $2$-system associated is 
$$\mU:=\{(x,x^2+a,x^4+b):x\in \F_{128},a,b\in \F_2\},$$
which is scattered by \cite[Theorem 3.14]{lia2024short}, is minimal (by \cite[Theorem 6.3]{rankminimal_ABNR}) and also rank-metric intersecting, by Theorem \ref{thm:sufficient}.

\medskip

Despite these differences, we will observe a similarity: regarding the values of the length, there exists a regime in which rank-metric intersecting codes cannot exist, a gray area, and a regime in which rank-metric intersecting codes are known to exist. This is similar for minimal codes. From this perspective, our work is similar to that carried out in \cite{rankminimal_ABNR}.
\end{remark}

\bigskip

\section{A geometric interpretation}\label{sec:geom}

We aim to determine a geometric interpretation for intersecting codes in the rank metric. The geometric perspective in coding theory is classical and has proven fruitful in the last decades. For example, the well-known MDS conjecture was formulated in geometric terms, specifically, as a problem about arcs in projective space \cite{segre1955curve}. Other classical structures in coding theory, such as covering codes and minimal codes, can also be interpreted geometrically, as saturating sets \cite{davydov2011linear} or strong blocking sets \cite{alfarano2022geometric} respectively. This geometric approach has also yielded significant insights in the rank-metric setting; see, for example, \cite{randrianarisoa2020geometric,sheekey2019linear,rankminimal_ABNR, bonini2023saturating,bartoli2024saturating}.

In light of the discussion on the rank metric in Section \ref{sec:background}, it is natural to think that the geometric analogue of intersecting codes in the rank metric will take the form of  $q$-systems (or linear sets) with a prescribed property.

\begin{definition}
Let $\mU$ be an $[n,k]_{q^m/q}$ system. We say that $\mU$ is $2$-spannable if there exist two $\Fqm$-linear hyperplanes $\mH_{1}, \mH_{2}$ of $\Fqmk$ such that
$$\mU =  \mH_{1} \cap \mU  +  \mH_{2} \cap \mU .$$
\end{definition}

One could easily generalize this property and define $j$-spannable $q$-systems, with $j \geq 1$. However, this generalization is unnecessary for the purposes of this work.

\begin{remark}\label{rmk:weight}
Note that if $\mU$ is a $2$-spannable $[n,k]_{q^m/q}$ system, then there exist two $\Fqm$-linear hyperplanes $\mH_{1}, \mH_{2}$ of $\Fqmk$ such that
$$n \leq \wt_{\mU}(\mH_{1})+\wt_{\mU}(\mH_{1}).$$
In particular, there exists a hyperplane $\mH$ such that 
$$\wt_{\mU}(\mH)\geq \frac{n}{2}.$$
Hence, if for every hyperplane $\mH$ 
we have $\wt_\mU(\mH) < \frac{n}{2}$,
then $\mU$ is not $2$-spannable. This is a reformulation of Theorem \ref{thm:sufficient} (by \eqref{eq:weight}).
\end{remark}

\begin{proposition}\label{prop:boundspannable}
Let $\mU$ be an $[n,k]_{q^m/q}$ system. If $n \leq 2k - 2$, then $\mU$ is $2$-spannable.
\end{proposition}

\begin{proof}
Let $(u_{1}, \dots, u_{n})$ be an $\Fq$-basis of $\mU$. We can then choose $\mH_{1}$ to be a hyperplane containing $u_{1}, \dots, u_{k-1}$ and $\mH_{2}$ a hyperplane containing $u_{k}, \dots, u_{n}$. With this choice of $\mH_{1}, \mH_{2}$, we have
$$\mU =  \mH_{1} \cap \mU  +  \mH_{2} \cap \mU ,$$
which implies that $\mU$ is $2$-spannable.
\end{proof}

\begin{theorem} \label{thm:rank_int_geo}
Let $\C$ be a nondegenerate $\nkdm$ code with generator matrix $G$. The following statements are equivalent:
\begin{enumerate}
\item The code $\C$ is rank-metric intersecting.
\item For all $A \in \GL(n, q)$, the code $\C \cdot A$ is Hamming-metric intersecting.
\item For all $A \in \GL(n, q)$, the multiset $\mS \subseteq \pgm$, obtained identifying the columns of $G \cdot A$ with points in the projective space, is not contained in the union of two hyperplanes.
\item The $[n,k]_{q^m/q}$ system $\mU$ corresponding to the code $\C$ is not $2$-spannable.
\end{enumerate}
\end{theorem}

\begin{proof}
$(1)\Rightarrow(2):$ by Proposition \ref{prop:rankinthammingint}, if $\C$ is rank-metric intersecting, then it is also Hamming-metric intersecting. By Proposition \ref{prop:equivalence}, any code of the form $\C \cdot A$ with $A \in \GL(n,q)$ is rank-metric intersecting. Consequently, any code of the form $\C \cdot A$ is Hamming-metric intersecting.

\noindent $(2)\Rightarrow(1):$ if $\C$ is not rank-metric intersecting, then there exist $c, c' \in \C$ such that $\sigma(c) \cap \sigma(c') = \{0\}$. Let $B \subset \Fqn$ be an $\Fq$-basis of $\sigma(c)$ and $B' \subset \Fqn$ be an $\Fq$-basis of $\sigma(c')$. Since $\sigma(c) \cap \sigma(c') = \{0\}$, it is clear that the vectors in $B\cup B'$ are linearly independent, and their union can be completed to form a basis $\mB$ of $\Fqn$. Let $A \in \GL(n, q)$ be the change-of-basis matrix from $\mB$ to the canonical basis of $\Fqn$. It is then clear that  $$\{i:(c \cdot A)_i\neq 0 \}\cap \{i:(c' \cdot A)_i\neq 0 \}=\varnothing,$$ so that the code $\C \cdot A$ is not Hamming-metric intersecting.

\noindent $(2)\Leftrightarrow(3):$ it follows from \cite[Theorem 2.6]{borello2025geometry}.

\noindent $(3)\Rightarrow(4):$  if $\mU$ is $2$-spannable, then it has a basis contained in two hyperplanes of $\Fqmk$, and there exists an invertible matrix $A \in \GL(n,q)$ such that the columns of $G \cdot A$ are the vectors of this basis. It is then clear that the set of points $\mS \subset \pgm$ obtained from these columns will be contained in two projective hyperplanes (the images of the corresponding hyperplanes in $\Fqmk$).

\noindent $(4)\Rightarrow(3):$ recall that $\mU$ is the  $\Fq$-span of the columns of $G$, so that the action of $\GL(n,q)$ leaves it invariant. If the multiset $\mS$ is contained in two hyperplanes of $\pgm$, then $\mU$ has a basis contained in two hyperplanes of $\Fqmk$, i.e., $\mU$ is $2$-spannable.
\end{proof}

\begin{remark}
    From Theorem \ref{thm:rank_int_geo} and Proposition \ref{prop:boundspannable} we get a different proof of the fact that the length of a rank-metric intersecting code of dimension $k$ has to be at least $2k-1$.
\end{remark}

\begin{remark}
It is important to note that the property of being $2$-spannable is not stable under inclusion (in coding theoretical language, extending a rank-metric intersecting code by adding a column to its generator matrix may not result in a rank-metric intersecting code). Indeed, if we take a Gabidulin code ${\rm Gab}_{5, 3}(1,\alpha,\alpha^2,\alpha^3,\alpha^4)$, where $\alpha^5=\alpha^2+1$ is a primitive element of $\F_{32}$, this is rank-metric intersecting and one can check that there is no way to add a column to the generator matrix that still results in a rank-metric intersecting code. This behavior is completely different from that of minimal codes.
\end{remark}

\begin{remark}\label{rmk:clubs}
 We can read the example of Remark \ref{rmk:nonnecessary} in geometric terms: $L_\mU$ in this case is a linear set of rank $4$ in ${\rm PG}(1,16)$ which is a $2$-\emph{club}, that is a linear set with only one point of weight $2$ and all the others of weight $1$. The $2$-system $\mU$ is clearly not $2$-spannable. As shown in \cite[Lemma 2.12]{de2016linear}, for $h\geq 3$, the linear set associated to the $[h,2,2]_{q^h/q}$ code $\mC$ generated by
 $$G=\begin{pmatrix}
     \alpha&\alpha^2&\ldots & \alpha^{h-2}& \alpha^{h-1}&0\\
     0&0&\ldots&0&1&\alpha
 \end{pmatrix},$$
where $\alpha$ is a primitive element of $\F_{q^h}$, is an $(h-2)$-\emph{club} in ${\rm PG}(1,q^h)$ for any $q$. This means that there is exactly one point of weight $h-2$ and all the others have weight $1$. The $q$-system associated, of rank $h$, is clearly not $2$-spannable. Hence $\mC$ is rank-metric intersecting.
 
Also, if we have a linear set of rank $2k$ in ${\rm PG}(k-1,q^{2k})$ with only one hyperplane of weight $k$ and all the others of smaller weight, then the $q$-system $\mU$ is not $2$-spannable and this would give an infinite class of rank-metric intersecting $[2k,k,k]_{q^{2k}/q}$ codes not satisfying the sufficient condition of Theorem \ref{thm:sufficient}. However, it is not clear if such linear sets exist for all $k>2$. 
\end{remark}

This leads naturally to the following questions:

\begin{question}
Can we construct other examples of rank-metric intersecting codes with minimum distance not exceeding half the length? In particular, are there infinite families of such codes with dimension greater than 
$2$?
\end{question}

\bigskip

\section{Bounds on the parameters}\label{sec:bounds}

In this section, we explore bounds on the parameters of rank-metric intersecting codes that can be derived from our geometric characterization.

\begin{lemma}\label{prop:peso}
Let $\mU$ be an $[n,k]_{q^m/q}$ system in $\mV := \Fqmk$. Then for any $\Fm$-linear subspace $\mM$ of codimension $s$, we have
\begin{equation} \label{peso}
\wt_\mU (\mM)\geq n - sm.    
\end{equation}
Hence, if $\mU$ is not $2$-spannable and $\mH$ is a hyperplane of $\mV$, 
$$\wt_\mU(\mH) \leq n - k.$$
\end{lemma}

\begin{proof}
The first inequality is a direct consequence of the Grassmann's formula. Indeed, we have
\begin{align*}
 \wt_\mU(\mM) & = \dim_{\Fq}(\mM \cap \mU) = \dim_{\Fq}(\mM) + \dim_{\Fq}(\mU) - \dim_{\Fq}(\mM + \mU)\\   
 & \geq \dim_{\Fq}(\mM) + \dim_{\Fq}(\mU) - \dim_{\Fq}(\mV)\\
 &=m(k-s)+n-mk.
\end{align*}
For the second inequality, suppose there exists a hyperplane $\mH'$ such that $\wt_\mU(\mH') \geq n - k + 1$, with $\mU$ not contained in $\mH'$. Then there are at most $k - 1$ $\Fq$-linearly independent vectors in another hyperplane, say $\bar \mH$. Therefore, we have
$$\mU =  \mU \cap \mH'  +  \mU \cap \bar \mH,$$
and thus $\mU$ is $2$-spannable.
\end{proof}

\begin{theorem}\label{thm:bounddistance}
Let $\C$ be a nondegenerate $\nkdm$ rank-metric intersecting code. Then
$$k \leq d \leq m.$$
\end{theorem}

\begin{proof}
The statement $d \leq m$ is trivial (because the rank cannot exceed $m$).

Let $G$ be a generator matrix of $\mC$. The second part of Lemma \ref{prop:peso} and Formula \eqref{eq:weight} implies that, for all $u\in \Fqm^\ast$,
\[\rk(uG)=n-\wt_{\mU_G}(\langle u\rangle^\perp)\geq n-(n-k)=k.\]
Hence, $d=\min_{u\in \Fqm^\ast} \rk(uG)\geq k$.
\end{proof}

\begin{corollary}
If $n = 2k - 1$ and $n \leq m$, then a nondegenerate $\nkdm$ rank-metric intersecting code is an MRD code.
\end{corollary}

\begin{proof}
When $n \leq m$, the Singleton bound is $k + d \leq n + 1$. By Theorem \ref{thm:bounddistance}, we have $k \leq d$, which combined with the Singleton bound yields $$2k \leq k + d \leq n + 1.$$ 
If $n = 2k - 1$, then both inequalities must be equalities, and the code is MRD.
\end{proof}

The following theorem provides nonexistence results for rank-metric intersecting codes.

\begin{theorem}\label{thm:existence}
Let $2\leq k\leq \left\lfloor \frac{m+1}{2}\right\rfloor$. There exists a nondegenerate $[n,k,>\frac{n}{2}]_{q^m/q}$ rank-metric intersecting code for any
$$ n\in \{m,\ldots,2m - 2k+1\}.$$
\end{theorem}

\begin{proof}
We will prove the statement in terms of $q$-systems.

Let $\mathcal{W}$ be an $[m,k]_{q^m/q}$ system that is scattered with respect to hyperplanes (for example that corresponding to a Gabidulin code of parameters $[m,k,m-k+1]_{q^m/q}$). Let $\overline{\mathcal{W}}$ be an $\Fq$-linear subspace of $\Fqmk$ of dimension $r\leq m(k-1)$ and such that $\mathcal{W}\cap \overline{\mathcal{W}}=\{0\}$. 

Set $\mU = \mathcal{W} \oplus \overline{\mathcal{W}}$. Then $n := \dim_{\Fq} \mU = m + r$.

We want to prove that $\mU$ is not $2$-spannable whenever $r\leq m-2k+1$.

Suppose $r\leq m-2k+1$ and assume that there exists a hyperplane $\mH$ such that $\wt_\mU(\mH) \geq \frac{n}{2}$. Then
\begin{align*}
w_\mathcal{W}(\mH) & = \dim_{\Fq}(\mH \cap \mathcal{W})  = \dim_{\Fq}(\mH \cap \mU \cap \mathcal{W}) \\& = \dim_{\Fq}(\mH \cap \mU) + \dim_{\Fq}(\mathcal{W}) - \dim_{\Fq}((\mH \cap \mU) + \mathcal{W})\\ 
& \geq \dim_{\Fq}(\mH \cap \mU) + \dim_{\Fq}(\mathcal{W}) - \dim_{\Fq}(\mU)\\ & \geq \frac{m+r}{2} + m - (m + r) =\frac{m - r}{2}\geq \frac{m-(m-2k+1)}{2}\geq k,
\end{align*}
which contradicts the fact that $\mathcal{W}$ is scattered with respect to hyperplanes. It follows $\wt_{\mU}(\mH)<\frac{n}{2}$ for every hyperplane $\mH$, so that $\mU$ is not $2$-spannable (see Remark \ref{rmk:weight}). Moreover, the weight of every codeword of the associated code is greater than $n-\frac{n}{2}=\frac{n}{2}$. 
\end{proof}

The proof of Theorem \ref{thm:existence} gives an explicit way to construct a nondegenerate rank-metric intersecting code with the above parameters. We illustrate this with an example. 

\begin{example}\label{exa:quasiMRD}
Let $m=5$ and $k=2$. Let $\mathcal{W}$ be the $[5,2]_{q^5/q}$ system spanned by the columns of 
$$G = \begin{pmatrix}
1 & \alpha & \alpha^2 & \alpha^3 & \alpha^4 \\
1 & \alpha^q & \alpha^{2q} & \alpha^{3q} & \alpha^{4q}
\end{pmatrix},$$
where $\alpha$ is a primitive element of $\F_{q^5}$. It is easy to observe that 
$$\mathcal{W}=\{(a,a^q)^T:a\in \F_{q^5}\}.$$
Now take any vector $(x,y)$ such that $y\neq x^q$. Then the $[6,2]_{q^5/q}$ code generated by
$$G = \begin{pmatrix}
1 & \alpha & \alpha^2 & \alpha^3 & \alpha^4 & x\\
1 & \alpha^q & \alpha^{2q} & \alpha^{3q} & \alpha^{4q}& y
\end{pmatrix}$$
is rank-metric intersecting. Moreover, its minimum distance is greater than $3$. By the Singleton bound, this is a $[6,2,4]_{q^5/q}$ code, which is a quasi-MRD code (see \cite{de2018weight}).
\end{example}

\begin{remark}
One can verify that 
\begin{itemize}
    \item for $m\leq 11$ and $k\leq \lfloor \frac{m+1}{2}\rfloor$,
    \item for $m\geq 12$ and $k=2$ or $\left\lceil \frac{m+2+\sqrt{(m+2)^2-16}}{4}\right\rceil< k\leq \lfloor \frac{m+1}{2}\rfloor$,
\end{itemize}
all codes of length $2m-2k$ obtained extending Gabidulin codes with columns chosen as in the proof of Theorem \ref{thm:existence} and Example \ref{exa:quasiMRD} are quasi-MRD codes of parameters $[2m-2k,k,m-k+1]_{q^m/q}$ and they are rank-metric intersecting.\\ In the same way, one can verify that 
\begin{itemize}
    \item for $m\leq 8$ and $k\leq \lfloor \frac{m+1}{2}\rfloor$,
    \item for $m\geq 9$ and $k=2$ or $\left\lceil \frac{m+3+\sqrt{(m+3)^2-16}}{4}\right\rceil< k\leq \lfloor \frac{m+1}{2}\rfloor$,
\end{itemize}
all codes of length $2m-2k+1$ obtained extending Gabidulin codes with columns chosen as in the proof of Theorem \ref{thm:existence} and Example \ref{exa:quasiMRD} are quasi-MRD codes of parameters $[2m-2k+1,k,m-k+1]_{q^m/q}$ and they are rank-metric intersecting. Quasi-MRD codes are objects introduced in \cite{de2018weight} and studied geometrically in \cite{marino2023evasive}. As pointed out in the last reference, the existence of nontrivial quasi-MRD codes is a widely open problem for general parameters and only few families are known.
\end{remark}

\begin{theorem}\label{Th:nonexistence}
If $\mU$ is an $[n,k]_{q^m/q}$ system with $n > 2m-3$, then $\mU$ is $2$-spannable.
\end{theorem}
\begin{proof}
We divide the proof in a number of cases.
\begin{itemize}
    \item  $\mathbf{n=2m.}$ First, observe that \( |L_\mU| \leq \frac{q^{2m} - 1}{q - 1} \). This implies that \( L_\mU \) cannot intersect all $\Fqm$-subspaces of codimension 2. Indeed, if \( L_\mU \) intersected every such subspace, it would be a so-called 2-blocking set, and it is known from~\cite{bose1966characterization} that 2-blocking sets have size at least \( q^{2m} + q^m + 1 \). Therefore, there exists an $\Fqm$-subspace $\mM$ of codimension $2$ in $\Fqmk$ whose intersection with $\mU$ is trivial. Let $\mH_1$ and $\mH_2$ be two distinct hyperplanes such that $\mM=\mH_1\cap \mH_2$; then 
    \begin{eqnarray*}
        \wt_\mU(\mH_i)&=&\dim_{\Fq}(\mH_i\cap \mU)=\dim_{\Fq}\mH_i+\dim_{\Fq}\mU-\dim_{\Fq}(\mH_i+\mU)\\
        &\geq& m(k-1)+2m-mk=m,
    \end{eqnarray*}
    for $i \in \{1,2\}$. Also,   $$(\mU\cap\mH_1)\cap(\mU\cap\mH_2)=\mU\cap(\mH_1\cap \mH_2)=\mU\cap\mM=\{0\},$$ 
    so that $\mU = \mU \cap \mH_1  \oplus \mU \cap \mH_2 $, and hence $\mU$ is $2$-spannable.
\vspace*{0.3 cm}
\item $\mathbf{n>2m.}$ The $\Fq$-space $\mU$ contains an $\Fq$-subspace $\mathcal{W}$ of dimension $2m$. Let $\mM$ be an $\Fqm$-subspace of codimension $2$ in $\Fqmk$ whose intersection with $\mathcal{W}$ is trivial (which exists for the same argument as before). Note that $\mathcal{W}\oplus \mM=\F_{q^m}^k$, so that $\mU+\mM=\F_{q^m}^k$.
Hence, by the Grassmann's formula, $$\wt_\mathcal{U}(\mM) = n - 2m.$$
Let $\mH_1$ and $\mH_2$ be two hyperplanes such that $\mM=\mH_1\cap\mH_2$. Then 
$$\wt_\mU(\mH_i) \geq n - m,$$ 
for $i \in \{1,2\}$.
Now, 
\begin{align*}
 \dim_{\Fq}( \mU \cap \mH_1  +  \mU \cap \mH_2 ) & =   \dim_{\Fq} (\mU \cap \mH_1) + \dim_{\Fq} (\mU \cap \mH_2) -\dim_{\Fq}  (\mU \cap \mH_1 \cap \mH_2)\\
 & = \wt_\mU(\mH_1)+ \wt_\mU(\mH_2)- \wt_\mU(\mM)\\
 & \geq n-m+n-m-n+2m=n.
\end{align*}
Consequently, $\mU = \mU \cap \mH_1  +  \mU \cap \mH_2$ and $\mathcal{U}$ is $2$-spannable.
\vspace*{0.3 cm}
\item $\mathbf{n=2m-1.}$ Let $\mH$ be a hyperplane of $\Fqmk$. From Lemma \ref{prop:peso}, we have $\wt_\mU(\mH) \geq m - 1$. Let $\mM$ be an $\Fqm$-subspace of codimension $2$ in $\Fqmk$ whose intersection with $\mU$ is trivial (which exists for the same argument as before). This implies that $(\mU\cap \mH_1)\cap(\mU\cap \mH_1)=\{0\}$ for all pairs of different hyperplanes $\mH_1,\mH_2$ containing $\mM$, so that the intersections of $\mU\setminus\{0\}$ with all hyperplanes containing $\mM$ form a partition of $\mU\setminus\{0\}$. If all $q^m + 1$ hyperplanes containing $\mM$ have weight $m - 1$ in $\mU$, then
\[
(q^{m-1} - 1)(q^m + 1) <  q^{2m - 1} - 1=|\mU| - 1,
\]
a contradiction. By easy cardinality arguments, there exists a unique  hyperplane $\mH$ containing $\mM$ of weight $m$ in $\mU$, and all the others have weight $m - 1$. Choose another hyperplane $\mH'$ through $\mathcal{M}$ different from $\mH$. We have that  $\mU=\mU\cap \mH+\mU\cap\mH'$ and $\mathcal{U}$ is $2$-spannable.
\vspace*{0.3 cm}
\item $\mathbf{n=2m-2.}$ Let $\mH$ be a hyperplane of $\Fqmk$. From Lemma \ref{prop:peso}, we have $\wt_\mU(\mH) \geq m - 2$. Let $\mM$ be an $\Fqm$-subspace of codimension $2$ in $\Fqmk$ whose intersection with $\mU$ is trivial (which exists for the same argument as before). Again, the intersections of $\mU\setminus\{0\}$ with all hyperplanes containing $\mM$ form a partition of $\mU\setminus\{0\}$. If there exists a hyperplane containing $\mM$ of weight $m$, then $\mU$ is $2$-spannable. Suppose instead that all $q^m + 1$ hyperplanes containing $\mM$ have weight $m - 1$ or $m - 2$ in $U$. Let $a$ and $b$ be the number of hyperplanes containing $\mM$ with weight $m - 1$ and $m - 2$ in $\mU$, respectively. Then the only solution of the system
\begin{equation}\label{System}
\left\{
\begin{array}{l}
a(q^{m-1} - 1) + b(q^{m-2} - 1) = q^{2m - 2} - 1 \\
a + b = q^m + 1,
\end{array}
\right.
\end{equation}
is $a = q + 1$ and $b = q^m - q$. Therefore, there exist at least $2$ hyperplanes containing $\mM$ of weight $m - 1$ in $\mU$, which proves that $\mU$ is $2$-spannable.
\end{itemize}
\end{proof}

\begin{remark}
Note that the same argument does not hold for $n=2m-3$, in general, since there are solutions to System \eqref{System} allowing hyperplanes of weight which is too small.     
\end{remark}

Combining Theorems \ref{thm:existence} and \ref{Th:nonexistence}, we can characterize the existence of nondegenerate rank-metric intersecting codes. First, recall that such a code is intersecting if and only if its associated $q$-system is not $2$-spannable. This condition implies that any nondegenerate rank-metric intersecting code of parameters $\nkdqm$ must necessarily satisfy $n < 2m - 2$. We have also established the existence of these codes for parameters $\nkdqm$ when $n \in \{2k-1,\ldots,2m - 2k+1\}$, provided $2 \leq k \leq \frac{m+1}{2}$. Consequently, a \emph{gray area} remains concerning the length $n$: specifically,
$$2m - 2k + 2 \leq n \leq 2m - 3.$$
Within this range, our current results do not definitively determine whether a nondegenerate rank-metric intersecting code exists. Interestingly, for $k=2$, this gray area is empty. The first instance where an open problem arises is for $k=3$ and $m=5$.

In this specific case, the gray area collapses to the set $\{6,7\}$. Let us assume the $\mathbb{F}_q$-dimension of the associated $q$-system $\mU$ is $7$. According to Lemma \ref{prop:peso}, every line within $\mU$ must have a weight of at least $2$. If, however, a line of weight $5$ exists in $\mU$, then $\mU$ is necessarily $2$-spannable.

Therefore, we proceed under the assumption that every line $\ell$ in $\mU$ has a weight of at most $4$. By \cite[Theorem 5.2]{lia2024short}, we know that there are exactly $q^2+1$ lines of weight $4$ in $\mU$. Consider one such line, $\ell$, and let $P$ be a point on $\ell$ that is not contained in $\mU$. If all other lines passing through $P$ had a weight of $2$ in $\mU$, this would lead to the contradiction:
\[1\cdot (q^4-1)+q^5\cdot q^7-1=(q^2-1)<|\mU|-1.\]
Consequently, there must exist at least one other line through $P$ with a weight of $3$ in $\mU$, which in turn implies that $\mU$ is $2$-spannable.

Concerning the case of length $6$, if there were a line of weight $4$ or $5$ in the associated $q$-system $\mU$, then simple counting arguments show that $\mU$ would be $2$-spannable. So, suppose that $\mU$ is not $2$-spannable and that all lines in $\mU$ have weight at most $3$. Then $\mU$ must be scattered. Indeed, if $\mU$ contains a point of weight $2$, say $P$, then there are $q^3 + q^2 + q + 1$ lines of weight $3$ in $\mU$ through $P$. Let $\ell$ be one such line, and let $Q$ be a point on $\ell$ of weight $1$ in $\mU$. By projecting $\mU$ from $Q$ onto a line $m$ not passing through $Q$, we obtain an $\F_q$-linear set of $\PG(1, q^5)$ of rank $5$ that has at least one point of weight $2$. Then, by \cite[Main Theorem]{de2022weight}, this projected linear set must contain at least one more point of weight $2$, which arises from another line, say $s$, through $Q$ of weight $3$ in $\mU$. However, this line $s$ would then intersect all of the remaining $q^3 + q^2 + q$ lines of weight $3$ through $P$, a contradiction.

In this remaining case, we conducted an exhaustive computational search for $q=2$, which yielded no positive results. Calculations were performed using the computational algebra system {\sc Magma} \cite{BCP97}, accumulating approximately 85 days of total CPU time on an Intel Core i7-5960X CPU (3.00GHz) computer. Our search methodology proceeded as follows:
\begin{itemize}
\item We began by considering an $q$-system $\mU'$ of $\mathbb{F}_q$-dimension $5$. Up to $P\Gamma L$-equivalence, and without loss of generality, such a system can be represented in one of three forms:
$$\{(x,x^q+Ax^{q^3}+Bx^{q^4},x^{q^2}+Cx^{q^3}+Dx^{q^4}) : x \in \mathbb{F}_{q^5}\}, \quad \text{or}$$
$$\{(x,x^q+Ax^{q^2}+Bx^{q^4},x^{q^3}+Cx^{q^4}) : x \in \mathbb{F}_{q^5}\}, \quad \text{or}$$
$$\{(x,x^{q^2}+Ax^{q^4},x^{q^3}+Bx^{q^4}) : x \in \mathbb{F}_{q^5}\},$$
for some $A,B,C,D\in \mathbb{F}_{q^5}$.
\item Next, we constructed an $q$-system $\mU$ of $\mathbb{F}_q$-dimension $6$ extending $\mU$. By Lemma \ref{prop:peso}, such a system must intersect each line non-trivially. Since the line $X_0=0$ has weight $0$ in $\mU'$, we can, without loss of generality, assume that 
$$\mU=\mU'+\langle(0,\alpha,\beta) \rangle_{\mathbb{F}_q},$$
where $(\alpha,\beta)\in \mathbb{F}_{q^5}^2\setminus \{(0,0)\}.$
\item For each combination of the possible forms of $\mU'$ and the choice of $(\alpha,\beta)$, we systematically checked whether there exist lines of weight $3$ in $\mU$ that intersect at a point external to $\mU$. The existence of such lines would directly establish the $2$-spannability of $\mU$.
\end{itemize}

The preceding discussion can thus be summarized in the following theorem.

\begin{theorem}
The length of a nondegenerate $[n,3,d]_{q^5/q}$ rank-metric intersecting code is either $5$, in which case the code is MRD and explicit constructions are known, or $6$. In the binary case, however, a nondegenerate $[6,3,3]_{32/2}$ rank-metric intersecting code does not exist.
\end{theorem}

We conclude this section by posing the following questions.

\begin{question}
Does a $[6,3,3]_{q^5/q}$ rank-metric intersecting code exist, for $q>2$?
\end{question}
\begin{question}
Are there rank-metric intersecting $\nkdqm$ codes for $2m-2k+2\leq n\leq 2m-3$?
\end{question}

\bigskip

\section{$(2,1)$-separating in the rank metric}\label{sec:separ}

This section focuses on the concept of $(2,1)$-separation in the rank metric, which is, to the best of our knowledge, a natural yet unexplored topic. In the Hamming metric, as highlighted in \cite{sagalovich2009separating,randriambololona20132}, this topic—and more broadly, the study of separating systems—has a well-established history and its roots can arguably be traced back to \cite{renyi1961random}. Notably, \cite{korner1995extremal} presents a particularly elegant interpretation framed within information theory using binary sequences: how many different points can one find in the $n$-dimensional binary Hamming space so that no three of them are on a line (three points in a metric space are on a line if they satisfy
the triangle inequality with equality)?

\begin{definition}
A $(2,1)$-\emph{separating code} over $\F_q$ of length $n$ is a subset $\C$ of $\F_q^n$
such that any pairwise distinct $x,y,z\in \C$ satisfy
\[{\rm d}(x,z)<{\rm d}(x,y)+{\rm d}(y,z),\]
where ${\rm d}$ is the Hamming distance in $\F_q^n$.  
\end{definition}

The above question may be formulated in following terms: what is the largest cardinality of a $(2,1)$-separating code over $\F_q$ of length $n$? If the code is linear, then the condition of being $(2,1)$-separating is equivalent to that of being Hamming-intersecting and the question is the classical question of finding Hamming-intersecting codes of large rate. 

We can reformulate the above definition in the rank-metric framework.

\begin{definition}
A $(2,1)$-\emph{separating rank-metric code} of length $n$ is a subset $\C$ of $\F_{q^m}^n$
such that any pairwise distinct $x,y,z\in \C$ satisfy
\[\rk(x-z)<\rk(x-y)+\rk(y-z),\]
where the rankn is defined as in Definition \ref{Def:rank}.
\end{definition}

In the linear case, the following holds: an $\nkdqm$ code is $(2,1)$-separating if and only if for any nonzero $c,c'\in \C$
\[\rk(c+c')<\rk(c)+\rk(c').\]

\begin{theorem}
Let $\C$ be an $\nkdqm$ code. If $\C$ is rank-metric intersecting, then $\C$ is $(2,1)$-separating.
\end{theorem}

\begin{proof}
If $\C$ is rank-metric intersecting, then, for any nonzero $c,c'\in \C$,
$$\dim(\sigma(c)+\sigma(c'))=\dim\sigma(c)+\dim\sigma(c')-\dim(\sigma(c)\cap\sigma(c'))<\rk(c)+\rk(c').$$ 
Now, $\sigma(c+c')\subseteq \sigma(c)+\sigma(c')$, so that
$$\rk(c+c')\leq \dim(\sigma(c)+\sigma(c'))<\rk(c)+\rk(c').$$
\end{proof}

\begin{remark}
The converse is not true. Actually, as we have already observed in Remark \ref{rmk:min-int}, the $[km,k,m]_{q^m/q}$ simplex codes are not intersecting for $k\geq 2$. However, they are $(2,1)$-separating: this is due to the fact that 
$$\rk(c+c')\leq m <2m=\rk(c)+\rk(c'),$$
for any nonzero $c,c'\in 
\C$.
\end{remark}

Note that the simplex codes are examples of rank-metric codes which are minimal. Actually, all minimal codes are $(2,1)$-separating.

\begin{theorem}
Let $\C$ be an $\nkdqm$ code. If $\C$ is minimal, then $\C$ is $(2,1)$-separating.
\end{theorem}

\begin{proof}
Let $c,c'$ two nonzero codewords. We have $\sigma(c+c')\subseteq \sigma(c)+\sigma(c')$. If $\rk(c+c')=\rk(c)+\rk(c')$, then 
$$\sigma(c+c')=\sigma(c)+\sigma(c'),$$
so that $\sigma(c)\subsetneqq \sigma(c+c')$. Since $c'\neq 0$,  $c+c'$ is not a multiple of $c$. Hence the code $\C$ is not minimal.
\end{proof}

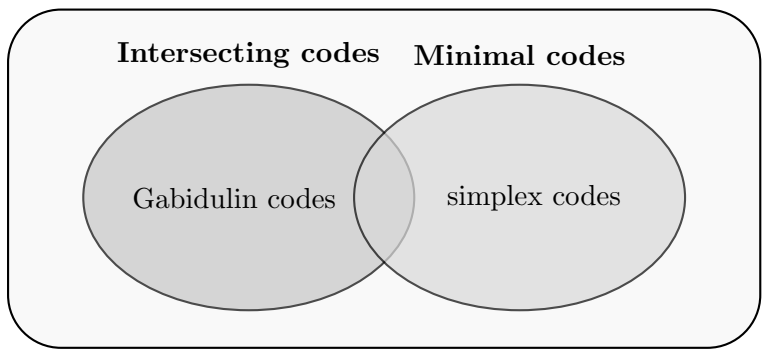
\begin{figure}[h]
\centering
\begin{tikzpicture}
    \draw[thick,rounded corners=20pt,fill=lightgray!10] (-5,-2) rectangle (5,2.5);
    \node at (0,2.9) {};
    \begin{scope}
        \clip (-5,-3) rectangle (5,3); 
        \draw[thick,fill=darkgray!30,opacity=0.7] (-1.8,0) ellipse (2.2 and 1.5);
        \node at (-1.8,0) {Gabidulin codes \, \,};
        \node at (-1.8,1.9) {\textbf{Intersecting codes}};
    \end{scope}
    \begin{scope}
        \clip (-5,-3) rectangle (5,3);  
        \draw[thick,fill=gray!30,opacity=0.7] (1.8,0) ellipse (2.2 and 1.5);
        \node at (1.8,0) {\, \, simplex codes};
        \node at (1.8,1.9) {\textbf{Minimal codes}};
    \end{scope}
\end{tikzpicture}
\vspace{-1.3cm}
\caption{\textbf{(2,1)-separating rank-metric codes}}
\label{fig:venn-rank}
\end{figure}

These findings are summarized in Figure \ref{fig:venn-rank}. When referring to Gabidulin codes in this context, we specifically mean those with dimension $k$ and length at least $2k-1$. If one considers the problem of determining the largest cardinality of a $(2,1)$-separating rank-metric code over $\mathbb{F}_{q^m}$ of length $n$, then rank-metric intersecting codes provide a partial answer. This is because, among all known $(2,1)$-separating rank-metric codes, the intersecting codes exhibit the highest rate.

\bigskip

\section{Frameproof codes}\label{sec:frameproof}

Frameproof codes were first introduced in 
\cite{boneh1998collusion} in the context of digital fingerprinting: in their paper, Boneh and Shaw present methods for assigning codewords to fingerprint digital content such as software, documents, music, and video. Fingerprinting involves uniquely marking each copy to enable traceability in case of unauthorized distribution, thus deterring misuse. A challenge arises when users collude: by comparing their copies, they may identify and manipulate the fingerprint to conceal their identities. Frameproof codes are proposed to build a general fingerprinting scheme resilient to collusion.

\begin{definition}
Let $\C\subseteq \F_{q}^n$ be a code (not necessarily linear) and $\mathcal{S}\subseteq \C$ a set of codewords. The \emph{set of descendants} of $\mS$ is 
$${\rm desc}(\mathcal{S})=\{w\in \F_{q}^n : \forall i\in \{1,\ldots,n\}, \exists c\in \mathcal{S}:w_i=c_i\}.$$
For $s\geq 2$, $\C$ is a $s$-\emph{frameproof code} if for any $\mathcal{S}\subseteq \C$ of size $|\mathcal{S}|\leq s$, it holds $${\rm desc}(\mathcal{S})\cap (\C\setminus\mathcal{S})=\varnothing.$$
\end{definition}

As explained in \cite{blackburn2003frameproof}, Boneh and Shaw employ a distinct definition of a descendant. The version of frameproof codes adopted here follows the formulation explicitly stated by Fiat and Tassa \cite{fiat1999dynamic}, who attribute its original introduction to Chor, Fiat, and Naor \cite{chor1994tracing}. For binary frameproof code constructions and discussions on their connections to related notions such as traceability codes and codes with the identifiable parent property, see the works by Stinson and Wei \cite{stinson1998combinatorial}, as well as Staddon, Stinson, and Wei \cite{staddon2001combinatorial}.

The following is a well-known remark.

\begin{remark}
If $s=2$, let $\mathcal{S}=\{c,c'\}$. Then
\begin{align*}
    {\rm desc}(\mathcal{S})&=\{w\in \F_{q}^n : \forall i\in \{1,\ldots,n\}, w_i=c_i \text{ or } w_i=c'_i\}\\
&=
\{w\in \F_{q}^n : \{i:w_i-c_i\neq 0\}\cap \{i:w_i-c'_i\neq 0\}=\varnothing\}.
\end{align*}
Consequently, a code $\mathcal{C}$ is defined as $2$-frameproof if, for any distinct codewords $c, c' \in \mathcal{C}$ and any codeword $x \in \mathcal{C} \setminus \{c,c'\}$, there exists an index $i \in \{1,\ldots,n\}$ such that $(x_i-c_i)(x_i-c'_i) \neq 0$. Notably, if $\mathcal{C}$ is a linear code, then the property of being $2$-frameproof is equivalent to being Hamming-metric intersecting.
\end{remark}

As discussed in \cite{blackburn2003frameproof}, even within the domain of $2$-frameproof codes, determining the largest achievable code rate for a given length is a problem of significant interest. This interest stems directly from applications such as digital fingerprinting, where a higher rate corresponds to a greater number of distinct fingerprints for a fixed length.

It is thus a natural extension to introduce analogous definitions and pose similar questions within the rank-metric setting. While the primary focus of this paper remains theoretical, these results could also find relevance in digital fingerprinting applications within the context of network coding, as suggested by works such as \cite{koetter2008coding} and \cite{silva2008rank}.

\begin{definition}
Let $\C\subseteq \F_{q^m}^n$ be a code (not necessarily linear) and $\mathcal{S}\subseteq \C$ a set of codewords. We define the \emph{set of rank-metric descendants} of $\mS$ as
$${\rm desc}_{\rk}(\mathcal{S})=\left\{w\in \F_{q^m}^n : \bigcap_{c\in\mathcal{S}}\sigma(w-c)=\{0\} \right\}.$$
For $s\geq 2$, $\C$ is $s$-rank-frameproof if for any $\mathcal{S}\subseteq \C$ of size $|\mathcal{S}|\leq s$, it holds that  ${\rm desc}_{\rk}(\mathcal{S})\cap (\C\setminus\mathcal{S})=\varnothing$.    
\end{definition}

\begin{theorem}
An $\nkdqm$ code is $2$-rank-frameproof if and only if it is rank-metric intersecting.
\end{theorem}

\begin{proof}
Let $\mathcal{S}=\{c,c'\}$. Then
$${\rm desc}_\rk(\mathcal{S})=\{w\in \F_{q^m}^n : \sigma(w-c)\cap \sigma(w-c')=\{0\} \}.$$
Therefore, $\C$ is $2$-rank-frameproof if  for any $\{c,c'\}\subseteq \C$ and for any $ x\in \C\setminus \{c,c'\}$, $$\sigma(x-c)\cap \sigma(x-c')\neq \{0\}.$$ If the code is linear, then this is equivalent to being rank-metric intersecting.    
\end{proof}
 
Finally, it is important to note that an $\nkdqm$ $2$-rank-frameproof code must satisfy the condition $n \geq 2k-1$. Furthermore, when $n \leq m$, Gabidulin codes serve as examples of codes achieving the largest possible rate.

We conclude this section by posing the following question.

\begin{question}
Is this definition of descendants relevant in the context of network coding for constructing a digital fingerprint?
\end{question}

\bigskip

\bibliographystyle{abbrv}
\bibliography{articles.bib}

\end{document}